\documentclass[11pt]{amsart}
\usepackage{amsmath,amssymb,latexsym,soul,cite}
\usepackage{color,enumitem,graphicx}
\usepackage[colorlinks=true,urlcolor=blue,
citecolor=red,linkcolor=blue,linktocpage,pdfpagelabels,
bookmarksnumbered,bookmarksopen]{hyperref}
\usepackage[english]{babel}
\usepackage[T1]{fontenc}
\usepackage[hyperpageref]{backref}

\usepackage[xetex,a4paper,left=2.62cm,right=2.62cm,top=2.6cm,bottom=2.6cm]{geometry}

\usepackage[colorinlistoftodos]{todonotes}

\makeatletter
\providecommand\@dotsep{5}
\def\listtodoname{List of Todos}
\def\listoftodos{\@starttoc{tdo}\listtodoname}
\makeatother

\numberwithin{equation}{section}

\newtheorem{thm}{Theorem}[section]
  \theoremstyle{plain}
  \newtheorem{lem}[thm]{Lemma}
  \theoremstyle{plain}
  \newtheorem{prop}[thm]{Proposition}
  \theoremstyle{plain}
  
  \theoremstyle{remark}
  \newtheorem{rem}[thm]{Remark}

\newcommand{\N}{{\mathbb N}}
\newcommand{\Z}{{\mathbb Z}}
\newcommand{\R}{{\mathbb R}}
\newcommand{\eps}{\varepsilon}

\title{Weyl-type laws for fractional $p$-eigenvalue problems}

\author[A.\ Iannizzotto]{Antonio Iannizzotto}
\author[M.\ Squassina]{Marco Squassina}
\address{Dipartimento di Informatica
\newline\indent
Universit\`a degli Studi di Verona
\newline\indent
C\'a Vignal II
\newline\indent
Strada Le Grazie I-37134 Verona, Italy}
\email{antonio.iannizzotto@univr.it}
\email{marco.squassina@univr.it}
\thanks{The second author was supported by 2009 MIUR project:
   ``Variational and Topological Methods in the Study of Nonlinear Phenomena''.}%e grazie a DOMINIQUE

\begin{document}

\begin{abstract}
We prove an asymptotic estimate for the growth of variational eigenvalues of 
fractional $p$-Laplacian eigenvalue problems on a smooth bounded domain.
\end{abstract}

\maketitle

%\listoftodos

\section{Introduction}

\noindent
Let $\Omega$ be a smooth bounded domain of $\R^N$ and, for $p>1$, consider the problem
\begin{equation*}
\left\{  \begin{array}{ll}
    -\Delta_p u=\lambda |u|^{p-2}u &\mbox{in $\Omega$,} \\
    u=0 &\mbox{on $\partial\Omega.$}
        \end{array}
      \right.
\end{equation*}
In the linear case $p=2$, the spectrum reduces to an increasing sequence $(\lambda_k)$ and 
a celebrated result obtained by Weyl \cite{W} around 1912 states
that the {\it counting function} $\mathcal{N}$ for eigenvalues, defined by
\begin{equation}\label{cf}
\mathcal{N}(\lambda)=\sharp\{k\in\N :\lambda_k<\lambda\},
\end{equation}
satisfies $\mathcal{N}(\lambda)\sim (2\pi)^{-N}\omega_N |\Omega|\lambda^{N/2}$
for $\lambda$ large, being $\omega_N$ the volume of the unit ball in $\R^N$.
In turn, the asymptotic growth of the $\lambda_k$'s is $k^{2/N}$, up to some constant depending on $N$ and $|\Omega|$.
In the case $p\neq 2$, although the spectrum is not yet completely understood, it is known that there exists
a sequence of {\em variational} eigenvalues $(\lambda_k)$ and, around 1989, Garc\'{\i}a Azorero $\&$ Peral Alonso \cite{GP} and Friedlander \cite{F} obtained the following asymptotic two-sided estimate for such sequence:
\[
C_1|\Omega|\lambda^{N/p}\leq\mathcal{N}(\lambda)\leq C_2|\Omega|\lambda^{N/p},
\,\,\quad\text{$\lambda>0$ large}.
\]
In this paper, we deal with the eigenvalue problem for the {\it fractional $p$-Laplacian}, namely
\begin{equation}\label{prob}
\left\{  \begin{array}{ll}
    (-\Delta)_p^s u=\lambda|u|^{p-2}u &\mbox{in $\Omega$,} \\
    u=0 &\mbox{in $\R^N\setminus\Omega,$}
        \end{array}
      \right.
\end{equation}
where $0<s<1$, $\Omega\subset\R^N$ ($N\geq 2$) is a bounded domain with Lipschitz boundary and
$(-\Delta)_p^s$ is defined, up to a normalization factor $c(s,p,N)$, as
\begin{equation*}
(-\Delta)^s_pu(x) = 2\lim_{\eps\to 0^+}\int_{\R^{N}\setminus B_\eps(x)}\frac{|u(x)-u(y)|^{p-2}(u(x)-u(y))}{|x-y|^{N+sp}}dy, \quad x\in\R^N.
\end{equation*}
In the particular but very important linear case $p=2$,  the operator $(-\Delta)_p^s$ reduces to
the linear fractional Laplacian $(-\Delta)^s$. Due to the non-local character of such operator, it is natural to work in the Sobolev space $W^{s,p}(\R^N)$ and express the Dirichlet condition on $\R^N\setminus\Omega$ rather than on $\partial\Omega$.
\vskip2pt
\noindent
% and it can be reformulated \cite{DPV}, up to the same normalization, by 
%\begin{equation*}
%(-\Delta)^su(x) = -\frac{1}{2}\int_{\R^N} \frac{u(x+y)+u(x-y)-2u(x)}{|y|^{N+2s}}dy, \quad x\in\R^N.
%\end{equation*}
Though fractional Sobolev spaces are well known since the beginning of the last century, especially in the field of harmonic analysis, they have become increasingly popular in the last few year, under the impulse of the work of Caffarelli $\&$ Silvestre \cite{CS} (see Di Nezza, Palatucci $\&$ Valdinoci \cite{DPV} and the reference within). The large amount of new contributions, especially focused on the linear case $p=2$, are motivated by several applications. For instance, Laskin \cite{L} has obtained, in quantum mechanics, a fractional generalization of the classical Schr\"odinger equation involving the operator $(-\Delta)^s$. The nonlinear eigenvalue problem~\eqref{prob} was first studied by Lindgren $\&$ Lindqvist \cite{LL} (for the case $p\ge 2$) and by Franzina $\&$ Palatucci \cite{FP} (for any $p>1$). In \cite{LL} much attention is paid to the asymptotics of problem \eqref{prob} as $p\to\infty$, while in \cite{FP} some regularity results for the eigenfunctions are proved.
\vskip2pt
\noindent
We provide a variational formulation for problem~\eqref{prob}. A {\it (weak) solution} of problem \eqref{prob} is a function $u\in W^{s,p}(\R^N)$ such that $u=0$ a.e.\ in $\R^N\setminus\Omega$ and
\begin{equation}\label{var}
\int_{\R^{2N}}\frac{|u(x)-u(y)|^{p-2}(u(x)-u(y))(v(x)-v(y))}{|x-y|^{N+sp}}dxdy=\lambda\int_\Omega |u|^{p-2}u v dx
\end{equation}
for all $v\in W^{s,p}(\R^N)$ such that $v=0$ a.e.\ in $\R^N\setminus\Omega$. We know that any solution is essentially bounded (see \cite[Theorem 3.2]{FP}), and H\"older continuous if $sp>N$ (see \cite[Theorem 3]{LL}). For all $\lambda\in\R$, there exists a non-zero solution $u$ of \eqref{prob}, then we say that $\lambda\in\R$ is an {\it eigenvalue} and $u$ 
is a $\lambda$-{\it eigenfunction}. The set of eigenvalues is the {\it spectrum} of \eqref{prob} and is denoted by $\sigma(s,p)$, and for all $\lambda\in\sigma(s,p)$ the set of $\lambda$-eigenfunctions is called $\lambda$-{\it eigenspace}. Clearly, $\sigma(s,p)\subset\R^+$ and all eigenspaces are star-shaped sets, as both sides of \eqref{prob} are $(p-1)$-homogeneous.
\vskip2pt
\noindent
We recall some remarkable properties of $\sigma(s,p)$:
\begin{itemize}
\item[$(i)$] $\sigma(s,p)$ is a closed set;
\item[$(ii)$] $\lambda_1=\min\sigma(s,p)>0$ is simple and isolated;
\item[$(iii)$] for all $\lambda\in\sigma(s,p)$ with $\lambda>\lambda_1$, any $\lambda$-eigenfunction $u$ is sign-changing in $\Omega$;
%\item[$(iv)$] there exists $K=K(s,p,N)$ such that for all $\lambda\in\sigma(s,p)$
%\[\lambda\geq K\max\{|\{u_\lambda>0\}|^{-ps/N},|\{u_\lambda<0\}|^{-ps/N}\};\]
\item[$(iv)$] if $(\Omega_j)$ is a non-decreasing sequence of domains such that $\Omega=\bigcup_{j=1}^\infty\Omega_j$, then $\lambda_1(\Omega_j)\searrow\lambda_1$  
(here $\lambda_1(\Omega_j)$ denotes the first eigenvalue of \eqref{prob} on the domain $\Omega_j\subset\Omega$);
\item[$(v)$] if $\Omega$ is a ball, then any positive (resp. negative) $\lambda_1$-eigenfunction is radially symmetric and radially decreasing (resp. increasing).
\end{itemize}
For the proofs of $(i)$-$(iv)$ and the exact ranges of $s,p$ for which these assertions hold,
see \cite{LL} and \cite{FP} (some of these properties also hold with a more general kernel $K(x,y)$, still with differentiability order $s$ and summability order $p$, replacing $|x-y|^{-N-sp}$). For $(v)$, see Proposition \ref{symmetry} below.
\vskip2pt
\noindent
In the present paper we focus on the higher fractional $p$-eigenvalues, following \cite{F} and \cite{GP}, dealing with the $p$-Laplacian operator. We will define a non-decreasing sequence 
$(\lambda_k)$ of variational (of min-max type) eigenvalues by means of the cohomological index (see Perera, 
Agarwal $\&$ O'Regan \cite{PAO}), and we will provide an 
estimate of the counting function of $(\lambda_k)$, still denoted by $\mathcal{N}(\lambda)$ and defined as in \eqref{cf}, at infinity.
\vskip2pt
\noindent
Our main result is the following:

\begin{thm}\label{main}
Let $0<s<1$, $p>1$, $N\geq 2$ and $\Omega\subset\R^N$ be a 
bounded domain with Lipschitz boundary. Then problem \eqref{prob} admits a non-decreasing 
sequence $(\lambda_k)$ of positive eigenvalues such that $\lambda_k\to\infty$ and
\begin{equation}
\label{asym-main1}
\mathcal{N}(\lambda)\geq C_1|\Omega|^\frac{sp}{Np-N+sp}\lambda^\frac{N}{Np-N+sp},
\,\,\,\quad\text{$\lambda>0$ large},
\end{equation}
for some constant $C_1>0$ depending only on $s$, $p$ and $N$. Furthermore, for $sp>N$,
\begin{equation}
\label{asym-main2}
\mathcal{N}(\lambda)\leq C_2|\Omega|^\frac{sp}{sp-N}\lambda^\frac{N}{sp-N},
\,\,\,\quad\text{$\lambda>0$ large},
\end{equation}
for some constant $C_2>0$ depending only on $s$, $p$ and $N$. 
\end{thm}

\noindent
Consequently, for $k$ large and $sp>N$, we have
$$
C'_1|\Omega|^{-\frac{sp}{N}}k^\frac{sp-N}{N}\leq \lambda_k\leq C'_2|\Omega|^{-\frac{sp}{N}}k^\frac{Np-N+sp}{N},
$$
for some positive constants $C'_i$ depending only on $s$, $p$ and $N$ ($i=1,2$).
We suspect that, actually, the following sharper Weil-type law holds
\begin{equation}
\label{asym}
\tilde C_1|\Omega|\lambda^{N/sp}\leq \mathcal{N}(\lambda)\leq \tilde C_2|\Omega|\lambda^{N/sp},
\,\,\quad\text{$\lambda>0$ large},
\end{equation}
for some positive constants $\tilde C_i$ depending only on $s$, $p$ and $N$ ($i=1,2$). Indeed, \eqref{asym} implies both \eqref{asym-main1} and \eqref{asym-main2} (at least if $sp>N$), and we have
\begin{align*}
|\Omega|^\frac{sp}{Np-N+sp}\lambda^\frac{N}{Np-N+sp} &\sim|\Omega|\lambda^\frac{N}{sp}, \quad \text{for $p$ close to $1$,} \\
|\Omega|^\frac{sp}{sp-N}\lambda^\frac{N}{sp-N} &\sim|\Omega|\lambda^\frac{N}{sp}, \quad \text{for $p$ large.}
\end{align*}
Non-optimality of our estimates may be explained as follows. In computing asymptotic estimates of variational einenvalues, a crucial step consists in proving sub- and super-additivity properties for the genus and co-genus of sublevels of the Sobolev norm on a domain $\Omega$ which is union of a disjoint family of open subsets $\Omega_i$. In the classical case of $p$-Laplacian problems ($s=1$), this is easily performed due to the following  splitting properties of Sobolev norms: if $\Omega=\Omega_1\cup\Omega_2$ 
and $\Omega_1\cap \Omega_2=\emptyset$, 
$$
\|u_1+u_2\|_{W^{1,p}_0(\Omega)}^p=\|u_1\|_{W^{1,p}_0(\Omega_1)}^p+\|u_2\|_{W^{1,p}_0(\Omega_2)}^p,
\,\,\quad  u_i\in W^{1,p}_0(\Omega_i)\,\, (i=1,2).
$$
In the fractional case ($0<s<1$), in general we have
$$
[u_1+u_2]_{s,p}^p\neq [u_1]_{s,p}^p+[u_2]_{s,p}^p,
\,\,\quad  \text{$u_i \in W^{s,p}(\R^N)$ with $u_i=0$ a.e.\ in $\R^N\setminus\Omega_i$\,\,\,$ (i=1,2)$},
$$
due to the nonlocal character of the Gagliardo norm. This forces us to introduce some correction multipliers, which eventually produce the asymmetric estimates \eqref{asym-main1}-\eqref{asym-main2}.
\vskip2pt
\noindent
In the linear case $p=2$, a completely different approach is possibel: the explicit asymptotic behaviour of eigenvalues was
obtained recently by Frank $\&$ Geisinger \cite{FG} and Geisinger \cite{G} and in the one-dimensional case by Kwasnicki \cite{K}. These results are consistent with \eqref{asym}.
%$$
%\lambda_k=\Big(\frac{\pi k}{2}-\frac{2\pi(1-s)}{8}\Big)^{2s}+{\mathcal O}\Big(\frac{1}{k}\Big),\qquad
%k\to\infty.
%$$
\vskip6pt
\noindent
The paper is organized as follows: in Section 2 we give a variational formulation of the problem and construct the sequence $(\lambda_k)$. In Section 3 we prove some technical lemmas on the Krasnoselskii genus and co-genus. In Section 4 we prove Theorem~\ref{main} (and $(v)$ above).

\medskip

\section{Construction of the variational eigenvalues}

\noindent
We first recall some basic notions from the Alexander-Spanier cohomology theory and introduce a cohomological index which goes back to Fadell $\&$ Rabinowitz \cite{FR}. Let $\mathcal{A}(X)$ denote the family of all nonempty, closed, symmetric subsets of a Banach space $X$, and for all $A\in\mathcal{A}(X)$, $B\in\mathcal{A}(X')$ we denote by $C_2(A,B)$ the set of all odd, continuous mappings $f:A\to B$. For all $A\in\mathcal{A}(X)$ we define the quotient space $\overline A=A/\Z_2$ and the classifying map $\varphi:\overline A\to\R P^\infty$ towards the infinite-dimensional projective space, which induces a homomorphism of cohomology rings $\varphi^*:H^*(\R P^\infty)\to H^*(\overline A)$. One can identify $H^*(\R P^\infty)$ with the polynomial ring $\Z_2[\omega]$ on a single generator $\omega$. Finally we define the {\it index} of $A$ as the positive integer
\[i(A)=\sup\{k\in\N : \ \varphi^*(\omega^{k-1})\neq 0\}.\]
We will not actually use much of index theory. All we need to know is that $i(S^{k-1})=k$ for all $k\in\N$ ($S^{k-1}$ denotes the unit sphere in $\R^k$, see \cite[Example 2.11]{PAO}) and that, if $A\in\mathcal{A}(X)$, $B\in\mathcal{A}(X')$ and $f\in C_2(A,B)$, then $i(A)\leq i(B)$ (see \cite[Proposition 2.12 $(i_2)$]{PAO}). We refer the reader to \cite{PAO} and to Motreanu, Motreanu $\&$ Papageorgiou \cite{MMP} for a detailed account of this subject.
\vskip2pt
\noindent
We also define the Krasnoselskii {\it genus} and {\it co-genus} by setting for all $A\in\mathcal{A}(X)$
\[\gamma^+(A)=\sup\{k\in\N : \ C_2(S^{k-1},A)\neq\emptyset\},\]
\[\gamma^-(A)=\inf\{k\in\N : \ C_2(A,S^{k-1})\neq\emptyset\}.\]
We have for all $A\in\mathcal{A}(X)$
\begin{equation}\label{ind-gen}
\gamma^+(A)\leq i(A)\leq\gamma^-(A).
\end{equation}
Indeed, for all $k\in\N$ for which there is a mapping $f\in C_2(S^{k-1},A)$, we have $i(A)\geq i(S^{k-1})=k$, hence $i(A)\geq\gamma^+(A)$. The second inequality is proved in a similar way.
\vskip2pt
\noindent
Now we turn to problem \eqref{prob} for which we provide a convenient variational formulation. For all measurable functions $u:\R^N\to\R$, we set
\[
\|u\|_{L^p(\R^N)}=\Big(\int_{\R^N}|u(x)|^p dx\Big)^\frac{1}{p},
\]
\[
[u]_{s,p}=\Big(\int_{\R^{2N}}\frac{|u(x)-u(y)|^p}{|x-y|^{N+sp}}dxdy\Big)^\frac{1}{p}.
\]
We define the fractional Sobolev space $W^{s,p}(\R^N)$ as the space of all functions $u\in L^p(\R^N)$ such that $[u]_{s,p}$
is finite and endow it with the norm
\[
\|u\|_{W^{s,p}(\R^N)}=\big(\|u\|_{L^p(\R^N)}^p+[u]_{s,p}^p\big)^\frac{1}{p}.
\]
We refer to \cite{DPV} for a description of fractional Sobolev spaces. Now we define a closed linear subspace of $W^{s,p}(\R^N)$:
\[
X(\Omega)=\left\{u\in W^{s,p}(\R^N):\, u=0 \,\, \mbox{a.e. in $\R^N\setminus\Omega$}\right\}.
\]
Clearly we can identify $\|\cdot\|_{L^p(\R^N)}$ and $\|\cdot\|_{L^p(\Omega)}$ on $X(\Omega)$. By using
\cite[Theorem 7.1]{DPV},  it is readily seen that the following Poincar\'e-type inequality holds:
\begin{equation}\label{poinc}
\|u\|_{L^p(\Omega)} \leq \lambda_1^{-\frac{1}{p}} [u]_{s,p},
\qquad\text{ for all $u\in X(\Omega)$ ($\lambda_1>0$).}
\end{equation}
Thus, we can equivalently renorm $X(\Omega)$ by setting $\|u\|_{X(\Omega)}=[u]_{s,p}$ for every $u\in X(\Omega)$. 
So, $(X(\Omega),\|\cdot\|_{X(\Omega)})$ is a uniformly convex (in particular, reflexive) Banach space. 
In fact, we have the linear isometry $F:X(\Omega)\to L^p(\R^{2N})$ defined, for all $u\in X(\Omega)$, by
\[
F(u)(x,y)=\frac{u(x)-u(y)}{|x-y|^{N/p+s}},
\,\,\,\quad (x,y)\in\R^{2N}.
\]
Whence, $F(X(\Omega))$ is uniformly convex as a linear subspace of $L^p(\R^{2N})$. 
Hence $X(\Omega)$ is uniformly convex too.
We denote by $X(\Omega)^*$ the topological dual of $X(\Omega)$ and we define a nonlinear operator 
$A:X(\Omega)\to X(\Omega)^*$ by setting for all $u,v\in X(\Omega)$
\[
\langle A(u),v\rangle=\int_{\R^{2N}}\frac{|u(x)-u(y)|^{p-2}(u(x)-u(y))(v(x)-v(y))}{|x-y|^{N+sp}}dxdy.
\]
Clearly $A$ is $(p-1)$-homogeneous and odd, a potential operator, satisfies for all $u,v\in X(\Omega)$
$$
\langle A(u),u\rangle=\|u\|_{X(\Omega)}^p,
\qquad
|\langle A(u),v\rangle|\leq \|u\|_{X(\Omega)}^{p-1}\|v\|_{X(\Omega)},
$$
hence by the uniform convexity of $X(\Omega)$
it enjoys the $(S)$-property, that is, whenever $(u_n)$ is a sequence in $X(\Omega)$ such that 
$u_n\rightharpoonup u$ in $X(\Omega)$ and $\langle A(u_n),u_n-u\rangle\to 0$, then 
$u_n\to u$ in $X(\Omega)$ (see\ \cite[Proposition 1.3]{PAO}).
\vskip2pt
\noindent
We set for all $u\in X(\Omega)$
\[I(u)=\|u\|_{L^p(\Omega)}^p, \quad J(u)=[u]_{s,p}^p.\]
Besides, we set
\[S=\big\{u\in X(\Omega):\, I(u)=1\big\}.\]
Clearly $I\in C^1(X(\Omega))$, hence $S$ is a $C^1$-Finsler manifold. Besides, $J\in C^1(X(\Omega))$ and for every $u,v\in X(\Omega)$
\[\langle J'(u),v\rangle=p\langle A(u),v\rangle.\]
We denote by $\tilde J$ the restriction of $J$ to $S$. For all $\lambda>0$, $\lambda$ is a critical value of $\tilde J$ if and only if it is an 
eigenvalue of \eqref{prob}. Indeed, if there exists $u\in S$ and $\mu\in\R$ such that $J(u)=\lambda$ 
and $J'(u)-\mu I'(u)=0$ in $X(\Omega)^*$, then for all $v\in X(\Omega)$ we have
\[
\int_{\R^{2N}}\frac{|u(x)-u(y)|^{p-2}(u(x)-u(y))(v(x)-v(y))}{|x-y|^{N+sp}}dxdy=\mu\int_\Omega |u(x)|^{p-2}u(x) v(x) dx,
\]
hence (taking $v=u$) $\lambda=\mu$. So, $u\neq 0$ satisfies \eqref{var}. Vice versa, 
if $\lambda$ is an eigenvalue of \eqref{prob}, then we can find 
a $\lambda$-eigenfunction $u\in X(\Omega)$ 
with $I(u)=1$. So, $u\in S$ is a critical point of $\tilde J$ at level $\lambda$ (see \cite[Proposition 3.54]{PAO}).
\vskip2pt
\noindent
Now we define the sequence $(\lambda_k)$. We denote by $\mathcal{F}$ the family of all nonempty, closed, symmetric subsets of $S$ and for all $k\in\N$ we set
\[\mathcal{F}_k=\{A\in\mathcal{F} : \ i(A)\geq k\}\]
and
\begin{equation}\label{min-max}
\lambda_k=\inf_{A\in\mathcal{F}_k}\sup_{u\in A} J(u)
\end{equation}
(this min-max formula differs from the classical ones by the use of the index in the place of the genus). Clearly, since $\mathcal{F}_{k+1}\subseteq\mathcal{F}_k$ for all $k\in\N$, the sequence $(\lambda_k)$ is non-decreasing. In particular (recalling that $J$ is even) we have
\[\lambda_1=\inf_{u\in S}J(u)=\inf_{u\in X(\Omega)\setminus\{0\}}\frac{[u]_{s,p}^p}{\|u\|_{L^p(\Omega)}^p},\]
hence $\lambda_1$ coincides with the first eigenvalue mentioned in the Introduction and in \eqref{poinc}.

\begin{prop}\label{ps}
The functional $\tilde J$ satisfies the Palais-Smale condition at any level $c\in\R$.
\end{prop}
\begin{proof}
Let $(u_n)$ and $(\mu_n)$ be sequences in $S$ and
$\R$ respectively, such that $J(u_n)\to c$ as $n\to\infty$
and $J'(u_n)-\mu_n I'(u_n)\to 0$ in $X(\Omega)^*$ as $n\to\infty$. 
Then, $(u_n)$ is bounded in $X(\Omega)$. Passing to a subsequence, 
we find $u\in X(\Omega)$ such that $u_n\rightharpoonup u$ in $X(\Omega)$ as $n\to\infty$
and $u_n\to u$ strongly in $L^p(\Omega)$ as $n\to\infty$, in light of~\cite[Theorem 7.1]{DPV}. 
In particular, $u\in S$. Moreover, 
\[\mu_n=J(u_n)+o(1)\to c.\]
Notice that, for all $n\in\N$, we have 
\begin{align*}
\left|\langle A(u_n),u_n-u\rangle\right| &= \left|\mu_n\int_\Omega |u_n(x)|^{p-2}u_n(x) (u_n(x)-u(x))dx\right|+o(1) \\
&\leq \left|\mu_n\right|\|u_n-u\|_{L^p(\Omega)}+o(1),
\end{align*}
and the latter vanishes as $n\to\infty$. 
Hence, by the $(S)$-property of $A$, we get $u_n\to u$ in $X(\Omega)$.
\end{proof}

\noindent
We have the following result for the sequence defined in \eqref{min-max}:

\begin{prop}\label{eigen}
For all $k\in\N$, $\lambda_k$ is an eigenvalue of problem \eqref{prob}. Moreover, $\lambda_k \to\infty$.
\end{prop}
\begin{proof}
We equivalently prove that $\lambda_k$ is a critical value of $\tilde J$, arguing by contradiction. Assume $\lambda_k$ is a regular value of $\tilde J$. Then, since $\tilde J$ satisfies the Palais-Smale condition by Proposition \ref{ps}, there exist a real $\eps>0$ and an odd homeomorphism $\eta:S\to S$ such that $J(\eta(u))\leq\lambda_k-\eps$ for all $u\in S$ with $J(u)\leq\lambda_k+\eps$ (see Bonnet \cite[Theorem 2.5]{B}). 
We can find $A\in\mathcal{F}_k$ such that $\sup_A J<\lambda_k+\eps$. Set $B=\eta(A)$, then $B\in\mathcal{F}$ and $i(B)\geq i(A)$, so $B\in\mathcal{F}_k$. We have for all $\sup_B J\leq\lambda_k-\eps$, which contradicts \eqref{min-max}.
\vskip2pt
\noindent
Finally, since $i(S)=\infty$ and $\sup_S J=\infty$, we easily draw $\lambda_k\to\infty$.
\end{proof}

\section{Preparatory results}

\noindent
We introduce some notation: for all $\Omega'\subset\R^N$ and for all $\lambda>0$, we set
\begin{align*}
M_0^\lambda(\Omega') & =\Big\{u\in X(\Omega'): \ \|u\|_{L^p(\Omega')}^p=1, \ [u]_{s,p}^p\leq\lambda\Big\}, \\
M^\lambda(\Omega')& =\Big\{u\in W^{s,p}(\R^N) : \ \|u\|_{L^p(\Omega')}^p=1, \ [u]_{s,p}^p\leq\lambda\Big\}.
\end{align*}
In order to prove our asymptotic estimate we need some information about the dependence of the genus and co-genus of sub-level sets of the types above, with respect to the domain and the level. We begin with a monotonicity property:

\begin{lem}\label{monot}
Assume that $\Omega\subseteq\Omega'$ and $0<\mu\leq\mu' $. Then
\[\gamma^+(M^\mu_0(\Omega))\leq \gamma^+(M^{\mu'}_0(\Omega')), \quad \gamma^-(M^\mu(\Omega))\leq \gamma^-(M^{\mu'}(\Omega')).\]
\end{lem}
\begin{proof}
The first inequality follows immediately from $M^\mu_0(\Omega)\subseteq M^{\mu'}_0(\Omega')$. Consider the mapping $f:M^\mu(\Omega)\to M^{\mu'}(\Omega')$ defined by 
$$
f(u)=\|u\|_{L^p(\Omega')}^{-1}u.
$$
Then, for every $u\in W^{s,p}(\R^N)$ with $\|u\|_{L^p(\Omega)}^p=1$ and $[u]_{s,p}^p\leq\mu$ we
have $\|f(u)\|_{L^p(\Omega')}=1$ and 
$$
[f(u)]_{s,p}^p=\|u\|_{L^p(\Omega')}^{-p}[u]_{s,p}^p\leq \|u\|_{L^p(\Omega)}^{-p}\mu\leq\mu'.
$$
Hence $f\in C_2(M^\mu(\Omega),M^{\mu'}(\Omega'))$, which proves the assertion.
\end{proof}

\noindent
We prove that the genus is (up to a correction factor) super-additive with respect to the domain:

\begin{lem}\label{genus-super}
If $\Omega_1,\ldots\Omega_m\subset\R^N$ are bounded domains with Lipschitz boundaries, such that $\Omega_i\cap\Omega_j=\emptyset$ for all $i\neq j$ and $\cup_{i=1}^m \overline{\Omega_i}=\overline\Omega$, then for all $\mu>0$
\[\sum_{i=1}^m \gamma^+(M^\mu_0(\Omega_i))\leq\gamma^+(M^{m^{p-1}\mu}_0(\Omega)).\]
\end{lem}
\begin{proof}
Avoiding trivial cases, we assume $\gamma^+(M^\mu_0(\Omega_i))=k_i\in\N$ and $f_i\in C_2(S^{k_i-1},M^\mu_0(\Omega_i))$ ($i=1,\ldots m$). Set $k=k_1+\ldots k_m$. For all $\xi\in S^{k-1}$ we set $\xi=(\xi_1,\ldots \xi_m)$ with $\xi_i\in\R^{k_i}$ and $|\xi_i|=t_i\in[0,1]$ ($i=1,\ldots m$). Clearly $t_1^2+\ldots t_m^2=1$. For all $1\leq i\leq m$ define $u_i\in X(\Omega_i)$ by setting
\[u_i=
\begin{cases}
f_i(\xi_i/t_i) & \text{if $t_i>0$,} \\
0 & \text{if $t_i=0$.}
\end{cases}\]
Hence $\|u_i\|_{L^p(\Omega_i)}$ is either $0$ or $1$ (according to either $t_i=0$ or $t_i>0$) and $[u_i]_{s,p}^p\leq\mu$. Set
\[f(\xi)=\sum_{i=1}^m t_i^\frac{2}{p}u_i.\]
Clearly $f(\xi)\in X(\Omega)$. Moreover
\[\|f(\xi)\|_{L^p(\Omega)}^p=\sum_{i=1}^m t_i^2\|u_i\|_{L^p(\Omega_i)}^p=\sum_{i=1}^m t_i^2=1\]
and a simple calculation shows
\[[f(\xi)]_{s,p}\leq\sum_{i=1}^m t_i^\frac{2}{p}[u_i]_{s,p}\leq\mu^\frac{1}{p}\sum_{i=1}^{m} t_i^\frac{2}{p}\leq m^\frac{p-1}{p}\mu^\frac{1}{p},\]
whence $[f(\xi)]_{s,p}^p\leq m^{p-1}\mu$. It is easily seen that the mapping $f:S^{k-1}\to M^{m^{p-1}\mu}_0(\Omega)$ is odd. Continuity is a more delicate matter. Let $(\xi^n)$ be a sequence in $S^{k-1}$ with $\xi^n\to\xi$ and denote $f(\xi^n)=u^n$, $f(\xi)=u$. Clearly $\xi^n_i\to\xi_i$ and $t^n_i\to t_i$ for all $1\leq i\leq m$ (with the obvious notation). So, for all $1\leq i\leq m$ one of the following cases occurs:
\begin{itemize}
\item if $t_i>0$, then $t^n_i>0$ for $n\in\N$ big enough and $u_i^n=f_i(\xi_i^n/t_i^n)$, so in $X(\Omega)$
\[\lim_n (t_i^n)^\frac{2}{p}u_i^n=\lim_n (t_i^n)^\frac{2}{p}f_i(\xi_i^n/t_i^n)=(t_i)^\frac{2}{p}f_i(\xi_i/t_i)=(t_i)^\frac{2}{p}u_i;\]
\item if $t_i=0$ and $t_i^n>0$ for $n\in\N$ big enough, then
\[(t_i^n)^\frac{2}{p}[u_i^n]_{s,p}\leq (t_i^n)^\frac{2}{p}\mu^\frac{1}{p},\]
and the latter tends to $0$ as $n\to\infty$, so $(t_i^n)^{2/p}u_i^n\to 0$ in $X(\Omega)$;
\item if $t_i=0$ and there exists a relabeled sequence such that $t_i^n=0$ for $n\in\N$ big enough, then clearly $(t_i^n)^{2/p}u_i^n=0$, and, reasoning as above, we conclude that $(t_i^n)^{2/p}u_i^n\to 0$ in $X(\Omega)$.
\end{itemize}
Thus, we have $u^n\to u$ in $X(\Omega)$, hence $f\in C_2(S^{k-1},M_0^{m^{p-1}\mu}(\Omega))$. Thus
\[\gamma^+(M_0^{m^{p-1}\mu}(\Omega))\geq k,\]
and the proof is concluded.
\end{proof}

\noindent
Now we prove that the co-genus is (up to a correction factor) sub-additive from the right:

\begin{lem}\label{cogenus-sub}
If $\Omega_1,\ldots\Omega_m\subset\R^N$ are bounded domains with Lipschitz boundaries, such that $\Omega_i\cap\Omega_j=\emptyset$ for all $i\neq j$ and $\cup_{i=1}^m \overline{\Omega_i}=\overline\Omega$, then for all $0<\mu'<\mu$
\[\gamma^-(M^\frac{\mu'}{m}(\Omega))\leq\sum_{i=1}^m \gamma^-(M^\mu(\Omega_i)).\]
\end{lem}
\begin{proof}
Avoiding trivial cases, for all $1\leq i\leq m$ we assume $\gamma^-(M^\mu(\Omega_i))=k_i\in\N$ and $f_i\in C_2(M^\mu(\Omega_i),S^{k_i-1})$. For all $1\leq i\leq m$ we define a mapping $\theta_i:M^{\mu'/m}(\Omega)\to\R\cup\{\infty\}$ by setting for all $u\in M^{\mu'/m}(\Omega)$
\[\theta_i(u)=
\begin{cases}
[u]_{s,p}^p/\|u\|_{L^p(\Omega_i)}^p & \text{if $\|u\|_{L^p(\Omega_i)}>0$,} \\
\infty & \text{if $\|u\|_{L^p(\Omega_i)}=0$.}
\end{cases}\]
Moreover, if $\|u\|_{L^p(\Omega_i)}>0$ we set $u_i=\|u\|_{L^p(\Omega_i)}^{-1}u$, so that $\theta_i(u)=[u_i]_{s,p}^p$. We have
\begin{equation}\label{minimum}
\min_{1\leq i\leq m}\theta_i(u)\leq\mu'.
\end{equation}
We prove \eqref{minimum} arguing by contradiction. Assume $\theta_i(u)>\mu'$ for all $1\leq i\leq m$, then
\[1=\|u\|_{L^p(\Omega)}^p=\sum_{i=1}^m \|u\|_{L^p(\Omega_i)}^p=\sum_{i=1}^m\frac{[u]_{s,p}^p}{\theta_i(u)}<\frac{m}{\mu'}[u]_{s,p}^p\]
(with the convention that $1/\infty=0$), a contradiction.
\vskip2pt
\noindent
We can find a mapping $\rho\in C^1(\R^+\cup\{\infty\})$ such that $\rho(t)=1$ for all $0\leq t\leq\mu'$, $\rho(t)=0$ for all $\mu\leq t\leq\infty$ and $0\leq\rho(t)\leq 1$ for all $t\in\R^+$. We set for all $u\in M^{\mu'/m}(\Omega)$
\[f(u)=\Big(\sum_{i=1}^m \rho(\theta_i(u))^2\Big)^{-\frac{1}{2}}\left(\rho(\theta_1(u))f_1(u_1),\ldots, \rho(\theta_m(u))f_m(u_m)\right)\]
(with the convention that $0\cdot\text{anything}=0$). By \eqref{minimum}, $f:M^{\mu'/m}(\Omega)\to S^{k-1}$ ($k=k_1+\ldots +k_m$) is well defined. Clearly $f$ is odd. We prove now that it is continuous. Let $(u^n)$ be a sequence in $M^{\mu'/m}(\Omega)$ such that $u^n\to u$ in $W^{s,p}(\R^N)$ for some $u\in M^{\mu'/m}(\Omega)$. For any $1\leq i\leq m$ one of the following cases occurs:
\begin{itemize}
\item if $\|u\|_{L^p(\Omega_i)}>0$, then $\|u^n\|_{L^p(\Omega_i)}>0$ for $n\in\N$ big enough, whence by continuity of $f_i$ we have $\rho(\theta_i(u^n))f_i(u^n_i)\to\rho(\theta_i(u)) f_i(u_i)$;
\item if $\|u\|_{L^p(\Omega_i)}=0$ and $\|u^n\|_{L^p(\Omega_i)}>0$ for all $n\in\N$, then $\|u^n\|_{L^p(\Omega_i)}^p\to 0$, so, recalling also that $[u^n]_{s,p}\to [u]_{s,p}>0$, we have
\[\lim_n\theta_i(u^n)=\lim_n\frac{[u^n]_{s,p}^p}{\|u^n\|_{L^p(\Omega_i)}^p}=\infty,\]
in particular $\rho(\theta_i(u^n))=0$ for $n\in\N$ big enough, so $\rho(\theta_i(u^n))f_i(u^n_i)\to 0$;
\item if $\|u\|_{L^p(\Omega_i)}=0$ and $\|u^n\|_{L^p(\Omega_i)}=0$ along a subsequence, then 
we can conclude that $\rho(\theta_i(u^n))f_i(u^n_i)=\rho(\theta_i(u)) g_i(u_i)=0$, and reasoning as above we get $\rho(\theta_i(u^n))f_i(u^n_i)\to 0$.
\end{itemize}
In any case, we have $f(u^n)\to f(u)$ as $n\to\infty$. Summarizing, $f\in C_2(M^{\mu'/m}(\Omega),S^{k-1})$. Thus
\[\gamma^-(M^{\mu'/m}(\Omega))\leq k,\]
and the proof is concluded.
\end{proof}

\noindent
Now we consider the behavior of the genus and co-genus in the presence 
of homothety:

\begin{lem} 
\label{scalat}
If $\tau>0$ and $\mu>0$, then 
\[\gamma^+(M^\mu_0(\Omega))=\gamma^+(M^\frac{\mu}{\tau^{sp}}_0(\tau\Omega)), \qquad \gamma^-(M^\mu(\Omega))=\gamma^-(M^\frac{\mu}{\tau^{sp}}(\tau\Omega)).\]
\end{lem}
\begin{proof}
For all $\tau>0$ and all $u\in W^{s,p}(\R^N)$ we set 
$u^\tau(z)=u(\tau^{-1}z)$, for all $z\in\R^N$. Then, a simple 
change of variables leads to
\begin{equation}\label{homo}
[u^\tau]_{s,p}^p=\tau^{N-sp}[u]_{s,p}^p, \quad \|u^\tau\|_{L^p(\tau\Omega)}^p=\tau^N\|u\|_{L^p(\Omega)}^p.
\end{equation}
For all $u\in M^\mu_0(\Omega)$ let us set $f(u)=\|u^\tau\|_{L^p(\tau\Omega)}^{-1}u^\tau$. 
Clearly $f(u)\in W^{s,p}(\R^N)$ and $f(u)=0$ a.e. in $\R^N\setminus\tau\Omega$. Furthermore, 
from equalities \eqref{homo}, we have $\|f(u)\|_{L^p(\tau\Omega)}=1$ and
\[[f(u)]_{s,p}^p=\frac{[u^\tau]_{s,p}^p}{\|u^\tau\|_{L^p(\tau\Omega)}^p}=\frac{[u]_{s,p}^p}{\tau^{sp}}\leq\frac{\mu}{\tau^{sp}}.\]
Thus, $f\in C_2(M^\mu_0(\Omega),M^{\mu/\tau^{sp}}_0(\tau\Omega))$. Since $f$ is a homeomorphism, we get the first equality.
\vskip2pt
\noindent
In a similar way, by using the homeomorphism $g\in C_2(M^{\mu/\tau^{sp}}(\tau\Omega),M^\mu(\Omega))$ defined for all $v\in M^{\mu/\tau^{sp}}(\tau\Omega)$ by setting $g(v)=\|v^{1/\tau}\|_{L^p(\Omega)}^{-1}v^{1/\tau}$, we achieve the second equality.
\end{proof}

\section{Proof of the main result}

\noindent
We give now the proof of Theorem \ref{main}.
\vskip2pt
\noindent
The first part of the assertion follows from Proposition \ref{eigen}, so we only need to prove the asymptotic estimates \eqref{asym-main1} and \eqref{asym-main2}. From \cite[Theorem 4.6 $(iii)$]{PAO}, for all $\lambda>0$ we have
\begin{equation}\label{count}
\mathcal{N}(\lambda)=i(M_0^\lambda(\Omega)).
\end{equation}
Preliminarly, we make some observations on cubes. Let $Q$ be a unit cube in $\R^N$ and $\lambda_0>0$ be such that $M^{\lambda_0}_0(Q)\neq\emptyset$. Then we have $\gamma^+(M^{\lambda_0}_0(Q))=r$ and $\gamma^-(M^{\lambda_0}(Q))=q$ for some $r,q\in\N$. For all $\lambda'>\lambda_0$ set $a_{\lambda'}=(\lambda_0/\lambda')^{1/sp}$. By Lemma \ref{scalat} (with $\mu=\lambda_0$ and $\tau=a_{\lambda'}$) we have
\begin{equation}\label{smallcube}
\gamma^+(M^{\lambda'}_0(a_{\lambda'}Q))=r, \quad \gamma^-(M^{\lambda'}(a_{\lambda'}Q))=q.
\end{equation}
Now we prove \eqref{asym-main1}. Since $\Omega$ is open, bounded and with a Lipschitz boundary, there exist $0<a<1$ and $n\in\N$ and a set $\Omega'\subseteq\Omega$, union of $n$ copies of $aQ$ with pairwise disjoint interiors, such that $na^N=|\Omega'|\geq|\Omega|/2$. We assume
\begin{equation}\label{big}
\lambda\geq\lambda_0 n^{p-1}a^{-sp},
\end{equation}
and set
\[C_1=2^{-\frac{N^2p-N^2+Nsp+sp}{Np-N+sp}} r \lambda_0^{-\frac{N}{Np-N+sp}}.\]
We consider the cube $aQ$ and set
\[\lambda'=\Big(\lambda_0^\frac{Np-N}{sp}(na^N)^{1-p}\lambda\Big)^\frac{sp}{Np-N+sp},\]
hence by \eqref{big} we have $\lambda'>\lambda_0$ and $a\geq a_{\lambda'}$. The cube $aQ$ contains the union of $m$ copies of $a_{\lambda'}Q$, where $m=[a/a_{\lambda'}]^N\geq 1$ (here $[\ \cdot \ ]$ denotes the integer part of a real number). From the elementary inequality $\alpha/2\leq[\alpha]\leq\alpha$ for all $\alpha\geq 1$ we have
\[2^{-N}\lambda_0^{-\frac{N}{sp}}a^N(\lambda')^\frac{N}{sp}\leq m\leq\lambda_0^{-\frac{N}{sp}}a^N(\lambda')^\frac{N}{sp}.\]
We apply the inequalities above, \eqref{smallcube} and Lemmas \ref{genus-super}, \ref{monot} and we have
\begin{align*}
2^{-N} r \lambda_0^{-\frac{N}{sp}}a^N(\lambda')^\frac{N}{sp} &\leq mr = m\gamma^+(M^{\lambda'}_0(a_{\lambda'}Q)) \\
&\leq \gamma^+(M^{m^{p-1}\lambda'}_0(aQ)) \leq \gamma^+\Big(M^{\lambda_0^\frac{N-Np}{sp}a^{Np-N}(\lambda')^\frac{Np-N+sp}{sp}}_0(aQ)\Big).
\end{align*}
The inequality above rephrases as the following:
\begin{equation}\label{big1}
\gamma^+(M^{n^{1-p}\lambda}_0(aQ))\geq\frac{r}{2^N}\lambda_0^{-\frac{N}{Np-N+sp}}n^\frac{N-Np}{Np-N+sp}a^\frac{Nsp}{Np-N+sp}\lambda^\frac{N}{Np-N+sp}.
\end{equation}
We apply again Lemmas \ref{genus-super}, \ref{monot} and \eqref{big1} and we obtain
\begin{align*}
\gamma^+(M^\lambda_0(\Omega)) &\geq \gamma^+(M^\lambda_0(\Omega')) \geq n\gamma^+(M^{n^{1-p}\lambda}_0(aQ)) \\
&\geq 2^{-N} r \lambda_0^{-\frac{N}{Np-N+sp}}(na^N)^\frac{sp}{Np-N+sp}\lambda^\frac{N}{Np-N+sp} \geq C_1|\Omega|^\frac{sp}{Np-N+sp}\lambda^\frac{N}{Np-N+sp}.
\end{align*}
By \eqref{ind-gen} and \eqref{count}, we have \eqref{asym-main1}.
\vskip2pt
\noindent
Now we prove \eqref{asym-main2}, under the hypothesis $sp>N$. We can find $0<b<1$, $h\in\N$ and the union $\Omega''\subset\R^N$ of $h$ copies of $bQ$ with pairwise disjoint interiors, such that $\Omega\subseteq\Omega''$ and $hb^N=|\Omega''|\leq 2|\Omega|$. We assume
\begin{equation}\label{small}
\lambda\geq \left(2^{N+1}h b^{sp}\right)^{-1}\lambda_0,
\end{equation}
and $\lambda''>\lambda$, and we set
\[C_2=2^\frac{Nsp+sp+N}{sp-N}q\lambda_0^{-\frac{N}{sp-N}}.\]
We focus on the cube $bQ$. Setting
\[\lambda'=\Big(2^{N+1}\lambda_0^{-\frac{N}{sp}}h b^N\lambda''\Big)^\frac{sp}{sp-N},\]
so by \eqref{small} we have $\lambda'>\lambda_0$ and $b\geq a_{\lambda'}$. So, $bQ$ is contained in the union of $k=([b/a_{\lambda'}]+1)^N$ copies of $a_{\lambda'}Q$ with pairwise disjoint interiors. From the elementary inequality $\alpha\leq[\alpha]+1\leq 2\alpha$ for all $\alpha\geq 1$ we have
\[\lambda_0^{-\frac{N}{sp}}b^N(\lambda')^\frac{N}{sp}\leq k\leq 2^N\lambda_0^{-\frac{N}{sp}}b^N(\lambda')^\frac{N}{sp}.\]
We use the inequalities above, \eqref{smallcube} and Lemmas \ref{monot} and \ref{cogenus-sub} (with $\mu=\lambda'$ and $\mu'=\lambda'/2$) to get
\begin{align*}
\gamma^-(M^{h\lambda''}(bQ)) &= \gamma^-\Big(M^{2^{-N-1}\lambda_0^\frac{N}{sp}b^{-N}(\lambda')^\frac{sp-N}{sp}}(bQ)\Big) \leq \gamma^-(M^{(2k)^{-1}\lambda'}(bQ)) \\
&\leq k\gamma^-(M^{\lambda'}(a_{\lambda'}Q)) = kq \leq 2^N q \lambda_0^{-\frac{N}{sp}} b^N(\lambda')^\frac{N}{sp},
\end{align*}
which rephrases as
\begin{equation}\label{small1}
\gamma^-(M^{h\lambda''}(bQ))\leq 2^\frac{Nsp+N}{sp-N} q \lambda_0^{-\frac{N}{sp-N}}h^\frac{N}{sp-N}b^\frac{Nsp}{sp-N}(\lambda'')^\frac{N}{sp-N}.
\end{equation}
Again by Lemmas \ref{monot} and \ref{cogenus-sub} (this time with $\mu=h\lambda''$ and $\mu'=h\lambda$) and by \eqref{small1}, we have
\begin{align*}
\gamma^-(M^\lambda(\Omega)) &\leq \gamma^-(M^\lambda(\Omega'')) \leq h \gamma^-(M^{h\lambda''}(bQ)) \\
&\leq 2^\frac{Nsp+sp+N}{sp-N} q \lambda_0^{-\frac{N}{sp-N}} |\Omega|^\frac{sp}{sp-N}(\lambda'')^\frac{N}{sp-N} = C_2|\Omega|^\frac{sp}{sp-N}(\lambda'')^\frac{N}{sp-N}.
\end{align*}
Letting $\lambda''\to\lambda$, we obtain
\[\gamma^-(M^\lambda(\Omega))\leq C_2|\Omega|^\frac{sp}{sp-N}\lambda^\frac{N}{sp-N},\]
which through \eqref{ind-gen} and \eqref{count} implies \eqref{asym-main2}. \qed

\bigskip
\noindent
Finally, we prove property $(v)$ stated in the Introduction.

\begin{prop}\label{symmetry}
If $\Omega$ is a ball, then any positive (resp. negative) $\lambda_1$-eigenfunction is radially symmetric and radially decreasing (resp. increasing).
\end{prop}
\begin{proof}
Let $u\in X(\Omega)$ be a positive $\lambda_1$-eigenfunction in the ball $\Omega$. If we denote $u^*$ the Schwartz symmetrization of $u$, we learn 
from Baernstein \cite[Theorem 3]{Ba} that $u^*\in X(\Omega)$, 
$\|u^*\|_{L^p(\Omega)}=\|u\|_{L^p(\Omega)}$ and $[u^*]_{s,p}\leq [u]_{s,p}$.
Hence, in turn, we have from \eqref{poinc}
\[\lambda_1\leq\frac{[u^*]_{s,p}^p}{\|u^*\|_{L^p(\Omega)}^p}\leq\frac{[u]_{s,p}^p}{\|u\|_{L^p(\Omega)}^p}=\lambda_1.\]
Thus $u^*\in X$ is a $\lambda_1$-eigenfunction too. By \cite[Theorem 4.2]{FP}, $u^*$ and $u$ are proportional and by the equalities above we obtain $u^*=u$, completing the proof.
\end{proof}

\begin{rem}\label{genus}
We observe that alternative sequences of variational eigenvalues $(\mu^\pm_k)$ can be produced by replacing the index $i$ with the genus/co-genus $\gamma^\pm$ in the min-max formula \eqref{min-max}  (see \cite[p. 75]{PAO}). Due to \eqref{ind-gen}, we then have $\mu^-_k\le\lambda_k\le\mu^+_k$ for all $k\in\N$, while it is not known whether the sequences coincide or not. In any case, denoting $\mathcal{N}^\pm$ the counting function for $(\mu^\pm_k)$, we have $\mathcal{N}^+(\lambda)\le\mathcal{N}(\lambda)\le\mathcal{N}^-(\lambda)$ for all $\lambda>0$, hence estimate \eqref{asym-main1} holds true for $\mathcal{N}^-$ and \eqref{asym-main2} for $\mathcal{N}^+$, respectively.
\end{rem}

\end{document}